\documentclass[12pt,a4paper]{amsart}
 \usepackage[T1]{fontenc}
\usepackage[top=35mm, bottom=35mm, left=30mm, right=30mm]{geometry}    
\usepackage{amssymb}
\usepackage{mathtools}

\usepackage[looser]{newpxtext}
\usepackage[smallerops]{newpxmath}

\usepackage{fancyhdr}
\makeatletter
\def\@@and{\MakeLowercase{and}}
\makeatother

\usepackage[hidelinks,pagebackref]{hyperref}

\usepackage[nobysame,msc-links]{amsrefs} 

\theoremstyle{definition}
\newtheorem{defn}{Definition}[section]
\newtheorem{exam}[defn]{Example}

\newtheorem{rem}[defn]{Remark}

\theoremstyle{plain}
\newtheorem{thm}[defn]{Theorem}
\newtheorem{lem}[defn]{Lemma}
\newtheorem{prop}[defn]{Proposition}
\newtheorem{coro}[defn]{Corollary}
\newtheorem{claim}[defn]{Claim}

\title[M\MakeLowercase{ean} L\MakeLowercase{i}-Y\MakeLowercase{orke chaos for a sequence of operators on} B\MakeLowercase{anach spaces}] 
{M\MakeLowercase{ean} L\MakeLowercase{i}-Y\MakeLowercase{orke chaos for a sequence of operators on} B\MakeLowercase{anach spaces}}

\author[J. L\MakeLowercase{i}]{J\MakeLowercase{ian} Li}
\address[J. Li]{Department of Mathematics,
	Shantou University, Shantou, 515821, Guangdong, China}
\email{lijian09@mail.ustc.edu.cn}
\urladdr{https://orcid.org/0000-0002-8724-3050}

\author[X. W\MakeLowercase{ang}]{X\MakeLowercase{insheng} Wang}
\address[X. Wang]{Department of Mathematics,
	Shantou University, Shantou, 515821, Guangdong, China}
\email{wangxs@stu.edu.cn}
\urladdr{https://orcid.org/0000-0002-5287-8902}

\author[J. Z\MakeLowercase{hao}]{J\MakeLowercase{ianjie} Zhao}
\address[Z. Zhao]{School of Mathematics, Hangzhou Normal University, Hangzhou, 311121, Zhejiang, China}
\email{zjianjie@hznu.edu.cn}
\urladdr{https://orcid.org/0000-0003-0790-7038}

	

\date{\today}

\begin{document}

\begin{abstract}
In this paper, we obtain the dichotomy for mean equicontinuity and mean sensitivity for a sequence of bounded linear operators from a Banach space to a normed linear space.
The mean Li-Yorke chaos for sequences and submultiplicative sequences of bounded linear operators are also studied.
Furthermore, several criteria for mean Li-Yorke chaos are established.
\end{abstract}

\maketitle

\section{Introduction}

It appears that universality is a generic phenomenon in analysis. 
In the survey \cite{G1999} Grosse-Erdmann discussed the theory of universal families including both of general theory and special universal families systematically. 
A universal family consists of continuous mappings $T_\lambda$ ($\lambda\in\Lambda$) between two topological spaces $X$ and $Y$, which has a universal element, i.e., $x\in X$ such that $\{T_\lambda x\colon\lambda\in\Lambda\}$ is dense in $Y$. 
In particular, the concept of hypercyclicity was discussed in \cite{G1999}, which is the study of the universality of a sequence of operators on a topological vector space. 
One special but important case is that when the sequence is generated by iterations of a single continuous operator, and such single continuous operator is called a hypercyclic operator.

A natural question is how to determine whether a sequence of operators or a single operator is hypercyclic. However it is not an easy task to make such a determination only using the definition of hypercyclicity. 
Hence it is necessary to establish some hypercyclicity criteria which give sufficient conditions under which the sequence of operators is hypercyclic and are easier to apply to specific examples. 
It is shown in \cite{BP1999} that a continuous linear operator $T$ on a Fr\'echet space satisfies the hypercyclicity criterion if and only if it is hereditarily hypercyclic, and if and only if $T\oplus T$ is hypercyclic.
For the case of sequence of operators, a more general result is obtained in \cite{BG2003}, which strengthens and extends the result in \cite{BP1999} and can be applied to an almost-commuting (refer to Section \ref{sec:LYsub} for the definition) sequence of operators.

When the family $\{T_\lambda\colon\lambda\in\Lambda\}$ forms a group under composition on a topological space, the universality is well known in topological dynamics under the name of topological transitivity, which is an important concept in the study of dynamics, and is related to another central concept --- chaos. 
Note that chaos gives a qualitative characterization of the complexity of systems and has been fully investigated. It is widely believed that Li-Yorke chaos is the first version of chaos, which was introduced by Li and Yorke in \cite{LY1975}.
However, the origin of this concept can be traced back even earlier, Poincar\'e has already observed chaotic phenomena.
With the need of research of dynamics, all kinds of chaos were established, such as Devaney chaos, mean Li-Yorke chaos, distributional chaos et al., refer to \cite{LY2016} and references therein for more information on this topic.
An interesting thing should be noted here that the operator which is Devaney chaotic must satisfy the hypercyclicity criterion given in \cite{BP1999}.

Now let us return to discuss the chaos theory concerning linear operators. The dynamics of linear systems is usually considered uncomplicated.
However, even for such linear systems, chaotic behavior can also been presented.
In 1990, Protopopescu reported in \cite{P1990} that a linear operator on an infinite dimensional linear space could be Devaney chaotic in one of his technical reports.
Moreover, various linear operators with all kinds of chaos including Li-Yorke chaos have been obtained, which promoted the investigation of the dynamics of general linear operators. 
In \cite{BBMP2011}, Li-Yorke and (uniform) distributional chaos for bounded operators on Banach spaces are studied. 
In \cite{BBMP2013}, a technical characterization of distributional chaos for bounded operators on an infinite dimensional separable Banach space is given. 
The mean Li-Yorke chaos of operators on Banach spaces is investigated in \cite{BBP2020}. Some examples are are provided in  \cite{BBP2020} that there exists distributionally chaotic operators that are not mean Li-Yorke chaotic, and it still open whether the converse implication holds (see Question 16 in \cite{BBP2020}).
See \cite{HSX2021} for a  generalization of mean Li-Yorke chaos for operators on Banach spaces. 
Using the techniques in topological dynamics, a uniform treatment of Li-Yorke chaos, mean Li-Yorke chaos and distributional chaos for continuous endomorphisms of completely metrizable groups is given in \cite{JL2022}, which is applied to the case of linear operators on Fr\'echet spaces, improving some results in the above references. 
Monographs \cite{BM2009} and \cite{GP2011} can be regarded as the introduction concerning this topic, where basic concepts and classical results related to chaos theory of linear operators can be found.

It is meaningful to consider the chaos theory of linear operators in a wide scope, such as the sequence of linear operators, which is not generated by a single continuous map. 
Recently, Conejero et al. studied distributional chaos for a family of operators on Fr\'echet spaces in \cite{CKMM2016}. In \cite{YHC2023}, the dichotomy of mean equicontinuity and mean sensitivity for non-autonomous linear systems is obtained.

In this paper, we will study mean Li-Yorke chaos for a sequence of bounded linear operators on Banach spaces, and generalize the results on a single operator to this setting.
We first obtain a  dichotomy of mean equicontinuity and mean sensitivity for
a sequence of operators.
Then we get some equivalent conditions for the existence of (dense) mean Li-Yorke scrambled sets.
Furthermore, under some conditions, we obtain more satisfied  equivalent conditions for the existence of (dense) mean Li-Yorke scrambled sets.

This paper is organized as follows. In Section~\ref{sec:me&ms}, we study the dichotomy for mean equicontinuity and mean sensitivity for sequences of bounded linear operators from a Banach space to a normed linear space. 
In Section~\ref{sec:MlY}, the mean Li-Yorke chaos for sequences of bounded linear operators are studied. 
In Section~\ref{sec:LYsub}, the mean Li-Yorke chaos for submultiplicative sequences of bounded linear operators are studied. 
Furthermore, some equivalent criteria for mean Li-Yorke chaos are established.

\section{Dichotomy for mean equicontinuity and mean sensitivity for a  sequence of operators}\label{sec:me&ms}

The concepts of  mean equicontinuity and mean sensitivity were first introduced in \cite{LTY2015} for dynamical systems on compact metric spaces.
In this section, we generalize the concepts of  mean equicontinuity and mean sensitivity for a sequence of operators. We will see in the next section that those concepts are closely related to mean Li-Yorke chaos.
\begin{defn}
    A sequence  $(T_i)_{i=1}^{\infty}$ of linear operators from a normed linear space $X$ to a normed linear space $Y$ is said to be
    \emph{mean equicontinuous} if for every $\varepsilon>0$ there exists $\delta>0$ such that for every $x,y\in X$ with $\Vert x-y\Vert<\delta$,
    \[
        \limsup_{n\to\infty}\frac{1}{n}\sum_{i=1}^{n}\Vert T_ix-T_iy\Vert<\varepsilon,
    \]
    and  \emph{mean sensitive} if  there is a constant $\delta>0$ such that for every $x\in X$ and $\varepsilon>0$ there exists some $y\in X$ with $\Vert x-y\Vert<\varepsilon$ and
    \[
        \limsup_{n\to\infty}\frac{1}{n}\sum_{i=1}^{n}\Vert T_ix-T_iy\Vert>\delta.
    \]
\end{defn}
We first study the equivalent characterizations of mean equicontinuity and mean sensitivity for general bounded linear operators, and obtain the following 
main result in this section.

\begin{thm}\label{thm:dich-mean-eq-mean-sen}
    Let $(T_i)_{i=1}^{\infty}$ be a sequence of bounded linear operators from a Banach space $X$ to a normed linear space $Y$. 
    Then either $(T_i)_{i=1}^{\infty}$ is mean equicontinuous or mean sensitive.
\end{thm}

We need the following definition of absolute Ces\`aro boundedness, which is introduced in~\cite{LH2015} for iterations of a linear operator.

\begin{defn}
    A sequence  $(T_i)_{i=1}^{\infty}$ of linear operators from a normed linear space $X$ to a normed linear space $Y$ is said to be \emph{absolutely Ces\`aro bounded} if there exists a constant $C>0$ such that for all $x\in X$,
    \[
        \sup_{n\in\mathbb{N}}\frac{1}{n}\sum_{i=1}^{n}\Vert T_i x\Vert\leq C\Vert x\Vert.
    \]
\end{defn}

Motivated by \cite{BBP2020}*{Theorem 4} and \cite{JL2022}*{Theorem 4.31}, we have the following result, which can be regarded as a mean
version of the uniform boundedness theorem for a sequence of linear operators.

\begin{thm}\label{thm:mean-UBT}
    Let $(T_i)_{i=1}^{\infty}$ be a sequence of bounded linear operators from a Banach space $X$ to a normed linear space $Y$.
    Then the following assertions are equivalent:
    \begin{enumerate}
        \item \label{thm1: ME}
              the sequence $(T_i)_{i=1}^{\infty}$  is mean equicontinuous;
        \item \label{thm2: x bounded}
             for all $x\in X$,
              \[
                  \sup_{n\in\mathbb{N}}\frac{1}{n}\sum_{i=1}^{n}\Vert T_i x\Vert<\infty;
              \]
        \item \label{thm3: aCb}
             the sequence $(T_i)_{i=1}^{\infty}$ is absolutely Ces\`aro bounded;
        \item \label{thm4: ME-sup}
             for every $\varepsilon>0$, there exists $\delta>0$ such that for every $x,y\in X$ with $\Vert x-y\Vert<\delta$,
              \[
                  \sup_{n\in\mathbb{N}}\frac{1}{n}\sum_{i=1}^{n}\Vert T_ix-T_iy\Vert<\varepsilon.
              \]

    \end{enumerate}
\end{thm}

\begin{proof}
     $ (\ref{thm1: ME})\Rightarrow (\ref{thm2: x bounded})$.  There exists $\delta>0$ such that for every $x\in X$ with $\Vert x\Vert<\delta$,
  \[
        \limsup_{n\to\infty}\frac{1}{n}\sum_{i=1}^{n}\Vert T_ix\Vert<1.
    \]
    Then for any $x\in X$,
    \[
        \limsup_{n\to\infty}\frac{1}{n}\sum_{i=1}^{n}\Vert T_ix\Vert = \frac{2\Vert x\Vert}{\delta}
        \limsup_{n\to\infty}\frac{1}{n}\sum_{i=1}^{n}\Vert T_i(\tfrac{\delta}{2\Vert x\Vert}x)\Vert < \frac{2\Vert x\Vert}{\delta}<\infty.
    \]
  
     $(\ref{thm2: x bounded}) \Rightarrow (\ref{thm3: aCb})$.
     For every $k\in\mathbb{N}$, let
    \[
        X_k=\biggl\{x\in X\colon  \sup_{n\in\mathbb{N}}\frac{1}{n}\sum_{i=1}^{n}\Vert T_i x\Vert\leq k\biggr\}.
    \]
    Then $X_k$ is a closed subset of $X$.
    As $X=\bigcup_{k\in \mathbb{N}} X_k$, by the Baire category theorem
    there exists $k_0\in\mathbb{N}$ such that the interior of $X_{k_0}$ is non empty.
    Pick a point $x_0\in X$ and $\delta_0>0$ such that $B(x_0,\delta_0)\subset X_{k_0}$, where
    $B(x_0,\delta_0)=\{x\in X\colon \Vert x-x_0\Vert<\delta_0\}$.
    For any $x\in X$, one has
    \[
        x_0+ \tfrac{\delta_0}{3\Vert x\Vert} x, x_0+ \tfrac{2\delta_0}{3\Vert x\Vert} x \in X_{k_0}.
    \]
    Let $C=\frac{6k_0}{\delta_0}$.
    Then for any $n\in\mathbb{N}$,
    \begin{align*}
        \frac{1}{n}
        \sum_{i=1}^{n}\Vert T_i x\Vert & =
        \frac{3\Vert x\Vert}{\delta_0}\frac{1}{n} \sum_{i=1}^{n}\Vert T_i (x_0+ \tfrac{\delta_0}{3\Vert x\Vert} x) - T_i (x_0+ \tfrac{2\delta_0}{3\Vert x\Vert} x)\Vert
        \\
         & \leq \frac{3\Vert x\Vert}{\delta_0}
        \biggl(\frac{1}{n}\sum_{i=1}^{n}\Vert T_i (x_0+ \tfrac{\delta_0}{3\Vert x\Vert} x)\Vert
        +\frac{1}{n}\sum_{i=1}^{n}\Vert T_i (x_0+ \tfrac{2\delta_0}{3\Vert x\Vert} x)\Vert\biggr) \\
        & \leq \frac{3\Vert x\Vert}{\delta_0}(k_0+k_0)
        \leq C \Vert x\Vert.
    \end{align*}
    This shows that the sequence $(T_i)_{i=1}^{\infty}$ is absolutely Ces\`aro bounded.

     $ (\ref{thm3: aCb} )  \Rightarrow   (\ref{thm4: ME-sup})$.  For every $\varepsilon>0$, pick $\delta=\frac{\varepsilon}{C}$, where $C$ is the constant in the definition of absolute Ces\`aro boundedness for $(T_i)_{i=1}^{\infty}$. Then for every $x,y\in X$ with $\Vert x-y\Vert<\delta$,
    \begin{align*}
        \sup_{n\in\mathbb{N}}\frac{1}{n}\sum_{i=1}^{n}\Vert T_ix-T_iy\Vert =
        \sup_{n\in\mathbb{N}}\frac{1}{n}\sum_{i=1}^{n}\Vert T_i(x-y)\Vert \\
        \leq C \Vert x-y\Vert <C\delta=\varepsilon.
    \end{align*}

     $ (\ref{thm4: ME-sup})  \Rightarrow (\ref{thm1: ME}) $.  
     It is clear.
\end{proof}

\begin{exam}
    Let $X$ be a Banach space and $I\colon X\to X$ be the identity operator.
    For each $i\in\mathbb{N}$, let $T_i=nI$ if $i=2^n$ for some $n\in\mathbb{N}$, and $T_i=I$ otherwise.
    It is easy to see that the sequence $(T_i)_{i=1}^\infty$ is absolutely Ces\`aro bounded but not norm-bounded.
\end{exam}

\begin{rem}
Examples of mixing operators that are absolutely Ces\`aro bounded can be found, for instsance, in \cite{BBMP2013}*{Example 5} and in  \cite{LH2015}*{Theorem 3.5}. Actually, Example 5 in \cite{BBMP2013} is a backward shift on $\ell^1(\mathbb{N})$ such that there is no point in $X$ has distributionally unbounded orbit (by the proof, this operator is in fact absolutely Ces\`aro bounded). By Theorem~\ref{thm:mean-UBT} those operators are also mean equicontinuous.
\end{rem}

From the opposite side, corresponding to Theorem \ref{thm:mean-UBT}, we have the following equivalent characterization of mean sensitivity. Some of the following equivalent characterizations are the same as \cite{BBP2020}*{Theorem 4}, but the proof method is different. Here we mainly use the characterization of mean equicontinuity in Theorem \ref{thm:mean-UBT}.

\begin{thm} \label{thm:mean-sensitive}
    Let $(T_i)_{i=1}^{\infty}$ be a sequence of bounded linear operators from a Banach space $X$ to a normed linear space $Y$.
    Then the following assertions are equivalent:
    \begin{enumerate}
        \item \label{thm1: mean sensitive}
              the sequence $(T_i)_{i=1}^{\infty}$  is mean sensitive;
        \item \label{thm2: x unbounded}
              there exists $x\in X$ such that
              \[
                  \limsup_{n\to{\infty}}\frac{1}{n}\sum_{i=1}^{n}\Vert T_i x\Vert=\infty;
              \]
        \item \label{thm3: residual points in X}
              the collection
              \[
                  \biggl \{ x\in X\colon \limsup_{n\to{\infty}}\frac{1}{n}\sum_{i=1}^{n} \Vert T_i x\Vert=\infty \biggr \}
              \]
              is residual in $X$;

        \item \label{thm4: residual points in X*X}
              the collection
              \[
                  \biggl\{ (x,y)\in X\times X\colon \limsup_{n\to{\infty}}\frac{1}{n}\sum_{i=1}^{n}\Vert T_i (x-y)\Vert=\infty \biggr \}
              \]
              is residual in $X\times X$;

        \item \label{thm5: bounded sequence}
              there exists a bounded sequence  $(y_k)_k$ in $X$ and a sequence $(N_k)_k$ in $\mathbb{N}$ such that
              \[
                  \sup_k\frac{1}{N_k}\sum_{i=1}^{N_k}\Vert T_i y_k\Vert=\infty;
              \]
        \item \label{thm6: sequence to 0}
              there exists a sequence  $(y_k)_k$ in $X$ and a sequence $(N_k)_k$ in $\mathbb{N}$ such that $\lim_{k\to\infty}y_k=0$ and
              \[
                  \liminf_{k\to\infty} \frac{1}{N_k}\sum_{i=1}^{N_k}\Vert T_i y_k\Vert>0.
              \]
    \end{enumerate}
\end{thm}

\begin{proof}
     $(\ref{thm1: mean sensitive}) \Rightarrow (\ref{thm2: x unbounded})$. 
    Otherwise, if for any $x\in X$,
    $$\limsup_{n\to \infty} \frac{1}{n} \sum_{i=1}^{n} \Vert T_ix\Vert<\infty,$$
    by Theorem \ref{thm:mean-UBT}, $(T_i)_{i=1}^{\infty}$ is mean equicontinuous. 
    Let $\delta$ be the constant in the definition of mean sensitivity.
    On the one hand, mean equicontinuity implies that there exists $\delta'>0$
    such that for every $x,y\in X$ with $\Vert x-y\Vert<\delta'$, one has
    \[
        \limsup_{n\to \infty} \frac{1}{n} \sum_{i=1}^{n} \Vert T_ix -T_iy\Vert <\delta.
    \]
    On the other hand, mean sensitivity implies that for every $x\in X$ and above $\delta'$,
    there exists some $y\in X$ with $\Vert x-y\Vert <\delta'$ such that
    \[
        \limsup_{n\to \infty} \frac{1}{n} \sum_{i=1}^{n} \Vert T_ix -T_iy\Vert >\delta.
    \]
    Which is a contradiction.
    
     $(\ref{thm2: x unbounded})  \Rightarrow (\ref{thm3: residual points in X})$. 
    Assume that $x_0\in X$ such that
    $$\limsup_{n\to \infty} \frac{1}{n}
        \sum_{i=1}^{n} \Vert T_ix_0\Vert=\infty.$$
    Let \[
        Y=\biggl\{ x\in X\colon \limsup_{n\to{\infty}}\frac{1}{n}\sum_{i=1}^{n}\Vert T_i x\Vert=\infty \biggr \}.
    \]
    It is clear that
    \[
        Y =\bigcap_{k=1}^\infty \biggl\{ x\in X\colon \exists n\geq k \text{ s.t. }
        \frac{1}{n}\sum_{i=1}^{n}\Vert T_i x\Vert >k\biggr\}.\]
    Then $Y$ is a $G_\delta$ subset of $X$.
    It suffices to show that for any $x\in X$ and $\varepsilon>0$, there exists some $y\in Y$ such that
    $\Vert x-y \Vert <\varepsilon$.
     
    If $x\in Y$, there's nothing to prove.
    If $x\notin Y$, put $y=x+\frac{\varepsilon}{2\Vert x_0\Vert}x_0 $,
    then $\Vert x-y \Vert <\varepsilon$
    and
    \begin{align*}
        \limsup_{n\to \infty} \frac{1}{n} \sum_{i=1}^n
        \Vert T_iy\Vert & =
        \limsup_{n\to \infty} \frac{1}{n} \sum_{i=1}^n \Vert T_i(x+\tfrac{x_0}{2\Vert x_0\Vert } \varepsilon )\Vert  \\
        & \ge \limsup_{n\to \infty} \frac{1}{n} \sum_{i=1}^n \Vert T_i( \tfrac{x_0}{2\Vert x_0\Vert } \varepsilon )\Vert -\limsup_{n\to \infty} \frac{1}{n} \sum_{i=1}^n \Vert T_i x \Vert \\
        & =\frac{ \varepsilon}{2\Vert x_0\Vert }  \limsup_{n\to \infty} \frac{1}{n} \sum_{i=1}^n \Vert T_i x_0  \Vert -\limsup_{n\to \infty} \frac{1}{n} \sum_{i=1}^n \Vert T_i x \Vert    \\
         & =\infty.
    \end{align*}

     $(\ref{thm3: residual points in X}) \Rightarrow (\ref{thm4: residual points in X*X})$.  It follows from the linearity of $(T_i)_i$.

     $ (\ref{thm4: residual points in X*X}) \Rightarrow (\ref{thm5: bounded sequence})$.  It is obvious.

     $(\ref{thm5: bounded sequence}) \Rightarrow (\ref{thm6: sequence to 0})$. 
    Assume that $(z_k)_k$ is the bounded sequence in $X$ and
    $(N_k)_k$ is the sequence in $\mathbb{N}$ such that
    \[
        \sup_k \frac{1}{N_k}\sum_{i=1}^{N_k} \Vert T_iz_k \Vert =\infty.
    \]
    Let $y_k=\frac{N_k}{ \sum_{i=1}^{N_k} \Vert T_iz_k \Vert} z_k$,
    then $\lim_{k\to \infty}y_k=0$ and
    \[
        \liminf_{k\to \infty} \frac{1}{N_k}\sum_{i=1}^{N_k} \Vert T_iy_k \Vert =1>0.
    \]

     $ (\ref{thm6: sequence to 0}) \Rightarrow (\ref{thm1: mean sensitive})$. 
    Suppose the sequences $(y_k)_k$ and $(N_k)_k$ are given by condition $(\ref{thm6: sequence to 0})$, and let
    \[
        \delta=\liminf_{k\to\infty} \frac{1}{N_k}\sum_{i=1}^{N_k}\Vert T_i y_k\Vert.
    \]
    Without loss of generality, we may assume that $(N_k)_k$ is increasing.
    For each $n\in \mathbb{N}$, let
    \[
        X_n=\biggl\{ x\in X: \exists k>n \ \text{s.t.} \ \frac{1}{k}\sum_{i=1}^{k} \Vert T_ix \Vert >
        \frac{\delta}{2} -\frac{1}{n} \biggr\}.
    \]
    Then each $X_n$ is an open subset of $X$.
    We will show that it is also dense in $X$.
    Let $U$ be a nonempty open subset of $X$ and pick $x\in U$.
    If $x\notin X_n$, then for any $k>n$, one has
    \[
        \frac{1}{k}\sum_{i=1}^{k}\Vert T_ix \Vert \le \frac{\delta}{2} -\frac{1}{n}.
    \]
    Since
    \[
        \liminf_{k\to \infty} \frac{1}{N_k} \sum_{i=1}^{N_k} \Vert T_iy_k \Vert =\delta,
    \]
    there exists $N\in \mathbb{N}$ such that for all $k>N$,
    we have
    \[
        \frac{1}{N_k} \sum_{i=1}^{N_k} \Vert T_iy_k \Vert > \delta -\frac{1}{2n}.
    \]
    As $\lim_{k\to \infty} (y_k+x) = x\in U$,
    there exists $k_0\in \mathbb{N}$ with $k_0>N$ such that $y_{k_0}+x \in U$
    and $N_{k_0}>n$.
    Then
    \begin{align*}
        \frac{1}{N_{k_0}}\sum_{i=1}^{N_{k_0}}\Vert T_i (y_{k_0}+x)\Vert
         & \ge \frac{1}{N_{k_0}}\sum_{i=1}^{N_{k_0}}\Vert T_i y_{k_0}\Vert  -\frac{1}{N_{k_0}}\sum_{i=1}^{N_{k_0}}\Vert T_i x\Vert \\
         & \ge \delta-\frac{1}{2n} - \Bigl(\frac{\delta}{2} - \frac{1}{n}\Bigr)                                                    \\
         & =\frac{\delta}{2} + \frac{1}{2n} > \frac{\delta}{2}-\frac{1}{n}.
    \end{align*}
    Which implies that $y_{k_0}+x\in U\cap X_n$.
    Therefore, the set
    \[
        X_0= \biggl\{x\in X: \limsup_{n\to \infty} \frac{1}{n}\sum_{i=1}^{n}\Vert T_ix \Vert \ge \frac{\delta}{2} \biggr\} =\bigcap_{n=1}^{\infty}X_n
    \]
    is residual in $X$.
    Fix $x\in X$ and $\varepsilon>0$,
    since $X_0$ is residual, there exists $z\in X$ with $\Vert z \Vert <\varepsilon$ and
    \[
        \limsup_{n\to \infty} \frac{1}{n}\sum_{i=1}^{n}\Vert T_iz \Vert \ge \frac{\delta}{2}.
    \]
    Let $y=z+x$, then $\Vert x-y \Vert <\varepsilon$ and
    \[
        \limsup_{n\to \infty} \frac{1}{n}\sum_{i=1}^{n}\Vert T_ix -T_iy\Vert =\limsup_{n\to \infty} \frac{1}{n}\sum_{i=1}^{n}\Vert T_iz \Vert \ge \frac{\delta}{2}.
    \]
    This shows that the sequence $(T_i)_{i=1}^{\infty}$ is mean sensitive.
\end{proof}

\begin{rem}
It is shown in \cite{BBPW2018}*{Theorem 25} that
there exists an operator $T$ on $X=\ell^1(\mathbb{N})$ and an invertible operator operator $T$ on $X=\ell^1(\mathbb{Z})$ such that $T$ is distributional chaotic and for every non-zero vector $x\in X$,
    \[
        \limsup_{n\to \infty} \frac{1}{n} \sum_{i=1}^{n} \Vert T^i x\Vert=\infty.
    \]
\end{rem}

Theorem~\ref{thm:dich-mean-eq-mean-sen} is a direct consequence of 
 Theorems~\ref{thm:mean-UBT} and~\ref{thm:mean-sensitive}.

\section{Mean Li-Yorke chaos for a  sequence of operators}\label{sec:MlY}

In this section, we study the mean Li-Yorke chaos for sequences of bounded linear operators from a Banach space to a normed linear space.

\begin{defn}
    Let $(T_i)_{i=1}^{\infty}$ be a sequence of bounded linear operators from a normed linear space $X$ to a normed linear space $Y$.
    A pair $(x,y)\in X\times X$ is called \emph{mean asymptotic} if
    \[
        \lim_{n\to \infty} \frac{1}{n} \sum_{i=1}^{n} \Vert T_ix-T_iy \Vert =0,
    \]
    and \emph{mean proximal} if
    \[
        \liminf_{n\to \infty} \frac{1}{n} \sum_{i=1}^{n} \Vert T_ix-T_iy \Vert =0.
    \]
    Let $\text{MAsym}(T_i)$ and $\text{MProx}(T_i)$ denote the collection of all mean asympotic pairs and mean proximal pairs, and are called \emph{mean asympotic relation} and \emph{mean proximal relation} respectively.
    For $x\in X$, the \emph{mean asymptotic cell} and \emph{mean proximal cell} of $x$ are defined by
    $$\text{MAsym}(T_i,x)=\{y\in X: (x,y)\in \text{MAsym}(T_i) \} $$
    and
    $$\text{MProx}(T_i,x)=\{y\in X: (x, y)\in \text{MProx}(T_i)\}$$
    respectively.
\end{defn}

The following lemmas are analogues of \cite{JL2022}*{Lemmas 4.1 and 4.8}, which will be used later.
\begin{lem}\label{mean proximal}
    Let $(T_i)_{i=1}^{\infty}$ be a sequence of bounded linear operators from a normed linear space $X$ to a normed linear space $Y$. Then
    \begin{enumerate}
        \item \label{Lem: MProx cell G-delta}
              for every $x\in X$, the mean proximal cell of $x$ is a $G_{\delta}$ subset of $X$;
        \item \label{Lem: MProx G-delta}
             the mean proximal relation of $(T_i)_i$ is a $G_{\delta}$ subset of $X\times X$;
        \item \label{Lem: cell-delta}
              for every $\delta>0$ and $x\in X$, the collection
              \[
                  \biggl\{ y\in X: \limsup_{n\to \infty} \frac{1}{n} \sum_{i=1}^n \Vert T_ix- T_iy \Vert \ge \delta \biggr\}
              \]
              is a $G_{\delta}$ subset of $X$;
        \item \label{Lem: delta}
              for every $\delta>0$, the collection
              \[
                  \biggl\{ (x,y)\in X\times X: \limsup_{n\to \infty} \frac{1}{n} \sum_{i=1}^n \Vert T_ix-T_iy \Vert \ge \delta \biggr\}
              \]
              is a $G_{\delta}$ subset of $X\times X$;
        \item \label{Lem: cell-infty}
              for every $x\in X$, the set
              \[
                  \biggl\{ y\in X: \limsup_{n\to \infty} \frac{1}{n} \sum_{i=1}^n \Vert T_ix-T_iy \Vert =\infty \biggl \}
              \]
              is a $G_{\delta}$ subset of $X$;
        \item \label{Lem: infty}
              the set
              \[
                  \biggl\{ (x, y)\in X\times X: \limsup_{n\to \infty} \frac{1}{n} \sum_{i=1}^n \Vert T_ix-T_iy \Vert =\infty \biggl \}
              \]
              is a $G_{\delta}$ subset of $X\times X$.
    \end{enumerate}
\end{lem}
\begin{proof}
    $(\ref{Lem: MProx cell G-delta})$ Note that for every $x\in X$,
    \[
        \text{MProx}(T_i,x)=\bigcap_{n=1}^{\infty}\biggl\{ y\in X: \exists k>n  \ \text{s.t.} \  \frac{1}{k} \sum_{i=1}^{k}\Vert T_ix-T_iy \Vert <\frac{1}{n} \biggr\},
    \]
    then by the continuity of $(T_i)_i$, we can obtain the set is a $G_{\delta}$ subset of $X$.

    $(\ref{Lem: cell-delta})$ Note that for every $\delta>0$ and $x\in X$,
    \begin{align*}
          & \biggl\{ y\in X: \limsup_{n\to \infty} \frac{1}{n} \sum_{i=1}^n\Vert T_ix-T_iy \Vert \ge \delta\biggr\} \\
        = & \bigcap_{n=1}^{\infty}\biggl\{ y\in X: \exists k>n  \ \text{s.t.} \  \frac{1}{k} \sum_{i=1}^{k}\Vert T_ix-T_iy \Vert >\delta-\frac{1}{n} \biggr\},
    \end{align*}
   then by the continuity of $(T_i)_i$, the set is a $G_{\delta}$ subset of $X$.

    $(\ref{Lem: cell-infty})$ Note that for every $x\in X$,
    \begin{align*}
          & \biggl\{ y\in X: \limsup_{n\to \infty} \frac{1}{n} \sum_{i=1}^n\Vert T_ix-T_iy \Vert = \infty \biggr\}  \\
        = & \bigcap_{n=1}^{\infty}\biggl\{ y\in X: \exists k>n  \ \text{s.t.} \  \frac{1}{k} \sum_{i=1}^{k}\Vert T_ix-T_iy \Vert >n \biggr\},
    \end{align*}
    then by the continuity of $(T_i)_i$, the set is a $G_{\delta}$ subset of $X$.
    
    The proofs of $(\ref{Lem: MProx G-delta})$, $(\ref{Lem: delta})$ and $(\ref{Lem: infty})$ are similar with that of $(\ref{Lem: MProx cell G-delta})$, $(\ref{Lem: cell-delta})$ and $(\ref{Lem: cell-infty})$ respectively.
\end{proof}

\begin{lem}\label{observations 1}
    Let $(T_i)_{i=1}^{\infty}$ be a sequence of bounded linear operators from a Banach space $X$ to a normed linear space $Y$. Then
    \begin{enumerate}
        \item \label{x-0}
              For every $x\in X$, $\text{MAsym}(T_i,x)=x+\text{MAsym}(T_i, {\bf 0})$ and
              $\text{MProx}(T_i,x)=x+\text{MProx}(T_i, {\bf 0})$;
        \item \label{MAsym dense}
              $\text{MAsym}(T_i, {\bf 0})$ is dense in $X$ if and only if
              $\text{MAsym}(T_i)$ is dense in $X\times X$;
        \item \label{MAsym residual}
              If MAsym$(T_i, {\bf 0})$ is residual in $X$, then MAsym$(T_i, {\bf 0})=X$;
        \item \label{MProx dense}
              $\text{MProx}(T_i, {\bf 0})$ is dense in $X$ if and only if
              $\text{MProx}(T_i)$ is dense in $X\times X$.
    \end{enumerate}
\end{lem}
\begin{proof}
    $(\ref{x-0})$ It follows from that $\Vert T_ix_1-T_ix_2 \Vert =\Vert T_i (x_1-x_2) \Vert$ for all $i\in \mathbb{N}$ and $x_1, x_2\in X$.
   
    $(\ref{MAsym dense})$ If $\text{MAsym}(T_i, {\bf 0})$ is dense in $X$, By $(\ref{x-0})$ for any $x\in X$, $\text{MAsym}(T_i, x)$ is also dense in $X$.
    Since $\text{MAsym}(T_i)=\cup_{x\in X} \{x\} \times \text{MAsym}(T_i, x)$, $\text{MAsym}(T_i)$ is dense in $X\times X$.
    Using the same method, we can obtain $(\ref{MProx dense})$.

    (\ref{MAsym residual}) Fix any $x\in X$, then
    \[
        \text{MAsym}(T_i, {\bf 0})-x=\{y-x\in X: y\in \text{MAsym}(T_i, {\bf 0})\}
    \]
    is also residual in $X$.
    Pick
    \[
        y\in \text{MAsym}(T_i, {\bf 0})\cap (\text{MAsym}(T_i, {\bf 0})-x).
    \]
    Then $(y, {\bf 0})$ and $(y+x, {\bf 0})$ are mean asymptotic.
    And the relation $ \text{MAsym}(T_i)$ is transitive, $(y, y+x)\in  \text{MAsym}(T_i)$
    and then $(x, {\bf 0})\in \text{MAsym}(T_i)$, i.e., $x\in \text{MAsym}(T_i, {\bf 0})$.
\end{proof}

 Now we present the relevant definitions for sequences of bounded linear operators $(T_i)_{i=1}^{\infty}$ from a normed linear space $X$ to a normed linear space $Y$. 
\begin{defn}
    Let $(T_i)_{i=1}^{\infty}$ be a sequence of bounded linear operators from a normed linear space $X$ to a normed linear space $Y$.
    A pair $(x,y)\in X\times X$ is called a \emph{mean Li-Yorke pair} if
    \[
        \liminf_{n\to \infty} \frac{1}{n} \sum_{i=1}^{n} \Vert T_ix-T_iy \Vert =0 \ \text{and} \
        \limsup_{n\to \infty} \frac{1}{n} \sum_{i=1}^{n} \Vert T_ix-T_iy \Vert >0.
    \]
    That is $(x,y)$ is mean proximal but not mean asymptotic.
    For a constant $\delta>0$, a pair $(x,y)\in X\times X$ is called a \emph{mean Li-Yorke $\delta$ pair} if
    \[
        \liminf_{n\to \infty} \frac{1}{n} \sum_{i=1}^{n} \Vert T_ix-T_iy \Vert =0 \ \text{and} \
        \limsup_{n\to \infty} \frac{1}{n} \sum_{i=1}^{n} \Vert T_ix-T_iy \Vert \geq \delta.
    \]
    A pair $(x,y)\in X\times X$ is called a \emph{mean Li-Yorke extreme pair} if
    \[
        \liminf_{n\to \infty} \frac{1}{n} \sum_{i=1}^{n} \Vert T_ix-T_iy \Vert =0 \ \text{and} \
        \limsup_{n\to \infty} \frac{1}{n} \sum_{i=1}^{n} \Vert T_ix-T_iy \Vert =\infty.
    \]
    A subset $K$ of $X$ is called \emph{mean Li-Yorke scrambled} (\emph{mean Li-Yorke $\delta$-scrambled}, \emph{mean Li-Yorke extremely scrambled}) if any two distinct points $x,y\in K$ form a mean Li-Yorke (mean Li-Yorke $\delta$, mean Li-Yorke extreme ) pair.

    We say that $(T_i)_{i=1}^{\infty}$ is \emph{mean Li-Yorke chaotic} (\emph{mean Li-Yorke $\delta$-chaotic}, \emph{mean Li-Yorke extremely chaotic}) if there exists an uncountable mean Li-Yorke scrambled (mean Li-Yorke $\delta$-scrambled, mean Li-Yorke extremely scrambled) subset of $X$.

    We say that $(T_i)_{i=1}^{\infty}$ is \emph{densely mean Li-Yorke chaotic} if there exists a densely uncountable mean Li-Yorke scrambled subset of $X$. Similarly, one can define \emph{densely mean Li-Yorke $\delta$-chaotic} and \emph{densely mean Li-Yorke extremely chaotic} for $(T_i)_{i=1}^{\infty}$.
\end{defn}

\begin{defn}
    Let $(T_i)_{i=1}^{\infty}$ be a sequence of bounded linear operators from a normed linear space $X$ to a normed linear space $Y$. A vector $x\in X$ is called \emph{absolutely mean semi-irregular} if
    \[
        \liminf_{n\to \infty} \frac{1}{n} \sum_{i=1}^{n} \Vert T_ix \Vert =0 \ \text{and} \
        \limsup_{n\to \infty} \frac{1}{n} \sum_{i=1}^{n} \Vert T_ix \Vert >0.
    \]
    and \emph{absolutely mean irregular} if
    \[
        \liminf_{n\to \infty} \frac{1}{n} \sum_{i=1}^{n} \Vert T_ix \Vert =0 \ \text{and} \
        \limsup_{n\to \infty} \frac{1}{n} \sum_{i=1}^{n} \Vert T_ix \Vert =\infty.
    \]
\end{defn}

\cite{BBP2020}*{Theorem 5} showed the case of iteration of an operator for the following results, and we now give the case of a sequence of bounded linear operators.

\begin{prop} \label{mean Li-Yorke-pair-vector}
    Let $(T_i)_{i=1}^{\infty}$ be a sequence of bounded linear operators from a normed linear space $X$ to a normed linear space $Y$. Then the following assertions are equivalent:
    \begin{enumerate}
        \item \label{Prop: mean Li-Yorke chaotic}
        $(T_i)_i$ is mean Li-Yorke chaotic;
        \item \label{mean Li-Yorke scrambled pair}
        there exists a mean Li-Yorke pair;
        \item \label{mean semi-irregular vector}
        there exists an absolutely mean semi-irregular vector.
    \end{enumerate}
\end{prop}

\begin{proof}
    $(\ref{Prop: mean Li-Yorke chaotic}) \Rightarrow (\ref{mean Li-Yorke scrambled pair})$ is trivial.

    $(\ref{mean Li-Yorke scrambled pair}) \Rightarrow (\ref{mean semi-irregular vector})$. 
    If $(x,y)$ is a mean Li-Yorke pair, then $u=x-y$ is an absolutely mean semi-irregular vector.

    $(\ref{mean semi-irregular vector}) \Rightarrow (\ref{Prop: mean Li-Yorke chaotic})$. If $x\in X$ is an absolutely mean semi-irregular vector, then it is easy to check that
    $\{\lambda u: \lambda\in \mathbb{K}\}$ is an uncountable mean Li-Yorke scrambled set.
\end{proof}

By the similar proof of Proposition \ref{mean Li-Yorke-pair-vector}, we have the following result.
\begin{prop}
    Let $(T_i)_{i=1}^{\infty}$ be a sequence of bounded linear operators from a normed linear space $X$ to a normed linear space $Y$. Then the following assertions are equivalent:
    \begin{enumerate}
        \item $(T_i)_i$ is mean Li-Yorke extremely chaotic;
        \item there exists a mean Li-Yorke extreme pair;
        \item there exists an absolutely mean irregular vector.
    \end{enumerate}
\end{prop}

We need the following topological tool, named Mycielski Theorem.
\begin{thm}\label{Mycielski Theorem}  \cite{Mycielski1964} 
    Let $X$ be a completely metrizable space without isolated points.
    If $R$ is a dense $G_\delta$ subset of $X\times X$, then there exists a $\sigma$-Cantor subset $K$ of $X$ such that for every two distinct points $x,y\in K$, $(x,y)\in R$.
    In addition, if $X$ is separable then we can require that the $\sigma$-Cantor set is dense in $X$.
\end{thm}

Now we have the following characterization of dense mean Li-Yorke extremely chaos for a sequence of bounded linear operators.

\begin{thm}\label{thm:denseLi-Yorke-delta-chaos}
    Let $(T_i)_{i=1}^{\infty}$ be a sequence of bounded linear operators from a separable Banach space $X$ to a normed linear space $Y$. 
    Then the following assertions are equivalent:
    \begin{enumerate}
        \item \label{Thm: densely mean LY delta-chaotic}
             $(T_i)_i$ is densely mean Li-Yorke $\delta$-chaotic for some $\delta>0$;
        \item \label{Thm: densely mean LY extremely chaotic}
             $(T_i)_i$ is densely mean Li-Yorke extremely chaotic;
        \item \label{Thm: densely set of mean LY ex sc pairs}
              $(T_i)_i$ has a dense set of mean Li-Yorke extreme pairs;
        \item \label{Thm: residual set of mean LY ex sc pairs}
               $(T_i)_i$ has a residual set of mean Li-Yorke extreme pairs;
        \item \label{Thm: dense set of ab mean irregualr vectors}
              $(T_i)_i$ has a dense set of absolutely mean irregular vectors;
        \item \label{Thm: residual set of ab mean irregualr vectors}
               $(T_i)_i$ has a residual set of absolutely mean irregular vectors;
        \item \label{Thm: MProx dense and mean sensitive}
              the mean proximal relation of $(T_i)_i$ is dense in $X\times X$ and $(T_i)_i$ is mean sensitive.
    \end{enumerate}
\end{thm}
\begin{proof}
      $ (\ref{Thm: residual set of mean LY ex sc pairs})  \Rightarrow (\ref{Thm: densely set of mean LY ex sc pairs})$, 
      $(\ref{Thm: densely mean LY extremely chaotic}) \Rightarrow (\ref{Thm: densely mean LY delta-chaotic})$ 
      and 
      $(\ref{Thm: residual set of ab mean irregualr vectors})\Rightarrow (\ref{Thm: dense set of ab mean irregualr vectors} )$. 
      Those are clear.
 
    $(\ref{Thm: densely set of mean LY ex sc pairs}) \Rightarrow (\ref{Thm: densely mean LY extremely chaotic})$. 
    Let $A$ be a dense set of mean Li-Yorke extreme pairs. Then $A$ can be expressed as the intersection of the mean proximal relation $\text{MProx}(T_i)$
    and the set
    \[
        \biggl\{(x,y)\in X\times X: \limsup_{n\to \infty} \frac{1}{n} \sum_{i=1}^n \Vert T_ix-T_iy \Vert =\infty \biggr\}.
    \]
    By Lemma \ref{mean proximal}, the set $A$ is a dense $G_{\delta}$ subset of $X\times X$.
    By Theorem \ref{Mycielski Theorem}, there exists a dense $\sigma$-Cantor subset $K$ of $X$ such that for every distinct points $x, y\in K$, $(x, y)\in A$.
    This implies that $K$ is a densely uncountable mean Li-Yorke extremely scrambled set,
    thus $(T_i)_i$ is densely mean Li-Yorke extremely  chaotic.

    $(\ref{Thm: densely mean LY delta-chaotic}) \Rightarrow (\ref{Thm: MProx dense and mean sensitive})$. Let $B$ be a densely uncountable mean Li-Yorke $\delta$-scrambled set for some $\delta>0$.
    Since $B\times B \subset \text{MProx}(T_i)$, the mean proximal relation of $(T_i)_i$ is dense in $X\times X$.

    Now put $\delta_0=\frac{\delta}{2}$. For every $x\in X$ and $\varepsilon>0$,
    as $B$ is dense and uncountable, choose $y_1 \neq y_2\in B$ with $\Vert x-y_1 \Vert <\varepsilon$ and
    $\Vert x-y_2 \Vert <\varepsilon$,
    and we have
    \[
        \limsup_{n\to \infty} \frac{1}{n} \sum_{i=1}^{n} \Vert T_iy_1-T_iy_2 \Vert \geq \delta.
    \]
    By the triangle inequality, one has either
    \[
        \limsup_{n\to \infty} \frac{1}{n} \sum_{i=1}^{n} \Vert T_ix-T_iy_1 \Vert \geq \delta_0
    \]
    or
    \[
        \limsup_{n\to \infty} \frac{1}{n} \sum_{i=1}^{n} \Vert T_ix-T_iy_2 \Vert \geq \delta_0.
    \]
    This implies that $(T_i)_i$ is mean sensitive.

    $(\ref{Thm: MProx dense and mean sensitive}) \Rightarrow (\ref{Thm: residual set of ab mean irregualr vectors})$. 
    The set of absolutely mean irregular vectors is the following intersection
    \[
        \text{MProx}(T_i, {\bf 0}) \cap \biggl\{x\in X: \limsup_{n\to \infty} \frac{1}{n} \sum_{i=1}^{n} \Vert T_ix \Vert =\infty\biggr\}.
    \]
    By Lemma \ref{mean proximal}, $\text{MProx}(T_i, {\bf 0})$ is a dense $G_{\delta}$ subset of $X$.
    Since $(T_i)_i$ is mean sensitive, by Theorem \ref{thm:mean-sensitive}, the set
    \[
        \biggl\{x\in X: \limsup_{n\to \infty} \frac{1}{n} \sum_{i=1}^{n} \Vert T_ix \Vert =\infty\biggr\}
    \]
    is also a dense $G_{\delta}$ subset of $X$. Then the set of absolutely mean irregular vectors is residual.

    $(\ref{Thm: dense set of ab mean irregualr vectors}) \Rightarrow (\ref{Thm: residual set of mean LY ex sc pairs})$.
    Let $C$ be a dense set of absolutely mean irregular vectors.
    Then $C\subset \text{MProx}(T_i, {\bf 0})$.
    By Lemma~\ref{lem:mprox-masym},  $\text{MProx}(T_i)$ is residual in $X\times X$.
    On the other hand every vector $x\in C$ satisfies
    \[
        \limsup_{n\to \infty} \frac{1}{n} \sum_{i=1}^{n} \Vert T_ix \Vert =\infty,
    \]
    by Theorem~\ref{thm:mean-sensitive},
    $(T_i)_i$ is mean sensitive and
    \[
        \biggl\{(x, y)\in X\times X:  \limsup_{n\to \infty} \frac{1}{n} \sum_{i=1}^{n} \Vert T_ix-T_iy \Vert =\infty\biggr\}
    \]
    is residual. Then the collection of  mean Li-Yorke extreme pairs is residual in $X\times X$.
\end{proof}

\begin{rem}\label{rem:separability-of-X}
In the proof of Theorem~\ref{thm:denseLi-Yorke-delta-chaos}, the separability of $X$ is only used in  $(\ref{Thm: densely set of mean LY ex sc pairs}) \Rightarrow (\ref{Thm: densely mean LY extremely chaotic})$ because  we want to apply the Theorem~\ref{Mycielski Theorem}.
Then 
(\ref{Thm: densely set of mean LY ex sc pairs})-(\ref{Thm: MProx dense and mean sensitive}) of Theorem~\ref{thm:denseLi-Yorke-delta-chaos}
are also equivalent in the condition that $X$ is a Banach space.
\end{rem}

\begin{exam}\label{exam:Ti-R}
    Let $X=\mathbb{R}$.
    \begin{enumerate}
        \item \label{Ex: 1}
        There exists a sequence $(T_i)_{i=1}^{\infty}$ of operators on $X$ such that for any non-zero $x\in X$,
              \[
                  \liminf_{n\to\infty}
                  \frac{1}{n}\sum_{i=1}^{n}\Vert T_i x\Vert=0 \text { and }
                  \limsup_{n\to\infty}
                  \frac{1}{n}\sum_{i=1}^{n}\Vert T_i x\Vert=\Vert x\Vert.
              \]

              Define $I\colon X\to X$, $x\mapsto x$ and $O\colon X\to X$, $x\mapsto 0$.
              For each $n\in\mathbb{N}$, define $a_n=2n!-1$ and $b_n=(n+1)!+n!-1$.
              Then $a_n<b_n<a_{n+1}$.
              For each $i\in\mathbb{N}$, define
              \[
                  T_i=\begin{cases}
                      O,  & i\in [a_n,b_n) \text{ for some } n      \\
                      2I, & i\in [b_n,a_{n+1}) \text{ for some } n.
                  \end{cases}
              \]
              Fix $x\in \mathbb{R}\setminus\{0\}$.
              For each $k\in\mathbb{N}$, let
              \[
                  f(k)= \frac{1}{k}\sum_{i=1}^k \Vert T_i x \Vert.
              \]
              It is easy to see that $f(k)$ is decreasing on each interval $[a_n,b_n)$,
              and increasing on each interval $[b_n,a_{n+1})$.
              Then
              \begin{align*}
                  \liminf_{k\to\infty} f(k) & = \lim_{n\to\infty} f(b_n-1)=  \lim_{n\to\infty}\frac{1}{b_n-1}\sum_{i=1}^{b_n-1} \Vert T_i x\Vert \\
               & = \lim_{n\to\infty}\frac{1}{b_n-1} \sum_{i=2}^n 2(a_i-b_{i-1})|x| \\
               &= \lim_{n\to\infty} \frac{2|x| (n!-1) }{ (n+1)!+ n!-1 -1} =0,
              \end{align*}
              and
              \begin{align*}
                  \limsup_{k\to\infty} f(k) & = \lim_{n\to\infty} f(a_{n+1}-1)
                  =  \lim_{n\to\infty}\frac{1}{a_{n+1}-1}\sum_{i=1}^{a_{n+1}-1} \Vert T_i x\Vert\\
                 & =\lim_{n\to\infty}\frac{1}{a_{n+1}-1} \sum_{i=2}^{n+1} 2(a_i-b_{i-1})|x| \\
                 & = \lim_{n\to\infty} \frac{2|x|  ( (n+1)!-1  ) }{ 2(n+1)!-1-1 }  =|x|.
              \end{align*}

        \item \label{Ex: 2}
         There exists a sequence $(T_i)_{i=1}^{\infty}$ of operators on $X$ such that for any non-zero $x\in X$,
              \[
                  \liminf_{n\to\infty}
                  \frac{1}{n}\sum_{i=1}^{n}\Vert T_i x\Vert=0  \ \text {and} \
                  \limsup_{n\to\infty}
                  \frac{1}{n}\sum_{i=1}^{n}\Vert T_i x\Vert=\infty.
              \]
              Let $c_1=1$, and choose $d_1, c_n, d_n, n\ge 2$ following the rules below:
              \[
                  d_n=c_n+n^3c_n, \ \text{and} \ c_{n+1}=d_n+n.
              \]
              For each $i\in\mathbb{N}$, define
              \[
                  T_i=\begin{cases}
                      O,         & i\in [c_n,d_n)\text{ for some } n;     \\
                      c_{n+1} I, & i\in [d_n,c_{n+1})\text{ for some } n.
                  \end{cases}
              \]
              By the same idea of $(\ref{Ex: 1})$, for every $x\in\mathbb{R}\setminus\{0\}$,
              \begin{align*}
                  \liminf_{n\to\infty}
                  \frac{1}{n}\sum_{i=1}^{n}\Vert T_i x\Vert
                   & = \lim_{n\to\infty} \frac{1}{d_n-1}
                  \sum_{i=1}^{d_n-1 }\Vert T_i x\Vert \\
                   & =\lim_{n\to\infty} \frac{|x| \biggr( c_2+2c_3+\cdots+(n-1)c_n   \biggr)}{d_n-1} \\
                   & \le \lim_{n\to\infty}   \frac{ n(n-1)c_n |x|}{c_n+n^3c_n-1} =0,
              \end{align*}
              and
              \begin{align*}
                  \limsup_{n\to\infty}
                  \frac{1}{n}\sum_{i=1}^{n}\Vert T_i x\Vert & = \lim_{n\to\infty} \frac{1}{c_{n+1}-1}\sum_{i=1}^{c_{n+1}-1} \Vert T_i x\Vert \\
                  & = \lim_{n\to\infty}  \frac{|x| \biggr( c_2+2c_3+\cdots+nc_{n+1}   \biggr)}{c_{n+1}-1} \\
                & \ge  \lim_{n\to\infty} n|x|=\infty.
              \end{align*}
    \end{enumerate}
\end{exam}

\begin{rem}
    Example~\ref{exam:Ti-R} $(\ref{Ex: 1})$ shows that there exists a sequence of operators on a finite dimensional space such that every non-zero vector is absolutely mean semi-irregular vector, but there is no absolutely mean irregular vectors.

    Example~\ref{exam:Ti-R} $(\ref{Ex: 2})$ shows that there exists a sequence of operators on a finite dimensional space such that every non-zero vector is absolutely mean irregular, and we mention here that in \cite{BBP2020} the authors constructed a forward shift satisfies this property.
\end{rem}

\begin{defn}
    Let $(T_i)_{i=1}^{\infty}$ be a sequence of bounded linear operators from a separable Banach space $X$ to a normed linear space $Y$. A vector subspace $Z$ of $X$ is called an \emph{absolutely mean irregular manifold} for $(T_i)_i$ if every non-zero vector $z\in Z$ is absolutely mean irregular.
\end{defn}

We have the following sufficient condition for the existence of dense irregular manifold for a sequence of operators, for the single operator case, refer to \cite{BBP2020}*{Theorem 29} and \cite{JL2022}*{Theorem 3.36} for Fr\'echet spaces.

\begin{thm}\label{thm:dense-mean-irregular-manifold}
    Let $(T_i)_{i=1}^{\infty}$ be a sequence of bounded linear operators from a separable Banach space $X$ to a normed linear space $Y$.
    If the  mean asymptotic relation of $(T_i)_i$ is dense in $X\times X$ and $(T_i)_i$ is mean sensitive, then $(T_i)_i$ has a dense absolutely mean irregular manifold.
\end{thm}
\begin{proof}
    Since $(T_i)_i$ is mean sensitive, by Theorem \ref{thm:mean-sensitive}, the set
    \[
        D=\biggl\{ x\in X\colon \limsup_{n\to{\infty}}\frac{1}{n}\sum_{i=1}^{n}\Vert T_i x\Vert=\infty \biggr \}
    \]
    is residual in $X$.
    We first show the following claim.
    \begin{claim} \label{$P(s_k)$ is residual}
        For any increasing sequence $(s_k)_k$ in $\mathbb{N}$,
        the set
        \[
            P(s_k)=\biggl\{x\in X: \liminf_{k\to \infty} \frac{1}{s_k} \sum_{i=1}^{s_k} \Vert T_i x \Vert=0 \biggr\}
        \]
        is a dense $G_{\delta}$ subset of $X$.
    \end{claim}
    \begin{proof}[Proof of the Claim]
        By Lemma \ref{observations 1}, $\text{MAsym}(T_i, {\bf 0})$ is dense in $X$.
        As $\text{MAsym}(T_i, {\bf 0}) \subset P(s_k)$, $P(s_k)$ is dense in $X$.
        Note that
        \[
            P(s_k)=\bigcap_{n=1}^{\infty} \biggl\{x\in X: \exists k>n  \ \text{s.t.} \  \frac{1}{s_k} \sum_{i=1}^{s_k} \Vert T_i x \Vert<\frac{1}{n}\biggr\},
        \]
        which implies that $P(s_k)$ is a dense $G_{\delta}$ subset of $X$.
    \end{proof}

    Fix a dense sequence $(z_m)_m$ in $X$.
    We will construct inductively a sequence $(x_m)_m \subset X$ such that
    $\Vert x_m-z_m \Vert <\frac{1}{m}$ and $\text{span} \{x_m: m\in \mathbb{N}\}$ is an absolutely mean irregular manifold.

    By the completeness of $X$, $X_1:=P(s_k)\cap D$ is a dense $G_{\delta}$ subset of $X$.
    We pick $x_1\in X_1$ with $\Vert x_1-z_1 \Vert <1$.
    Then there exist two increasing sequences $(s_k^{(1,1)})_k\subset (s_k)_k$
    and $(t_k^{(1)})_k$ in $\mathbb{N}$ such that
    \[
        \lim_{k\to \infty} \frac{1}{s_k^{(1,1)} } \sum_{i=1}^{s_k^{(1,1)}} \Vert T_{i}x_1 \Vert =0 \
        \text{and} \
        \lim_{k\to \infty} \frac{1}{t_k^{(1)}} \sum_{i=1}^{t_k^{(1)}} \Vert T_{i}x_1 \Vert =\infty.
    \]
    It is clear that $\mathbb{K}x_1$ is an absolutely mean irregular manifold.
    Assume that $x_m\in X$, sequences $(s_k^{(m,j)})_k$ and $(t_k^{(m)})_k$ in $\mathbb{N}$ have been construted for $m=1, \ldots,n$ and $j=1,\ldots, m$ such that
    \begin{enumerate}
        \item $\Vert x_m-z_m \Vert <\frac{1}{m}$ for $m=1, \ldots n$;
        \item for $m=2,\ldots, n$ and $j=1,\ldots, m-1$,
              $(s_k^{(m,j)})_k$ is a subsequence of $(s_k^{(m-1,j)})_k$, and $(s_k^{(m,i)})_k$ is a subsequence of $(t_k^{(m-1)})_k$;
        \item for $m=1, \ldots, n$ and $j=1, \ldots, m$,
              \[
                  \lim_{k\to \infty} \frac{1}{s_k^{(m,j)}} \sum_{i=1}^{ s_k^{(m,j)} } \Vert T_{i} x_m \Vert =0 \
                  \text{and} \
                  \lim_{k\to \infty} \frac{1}{t_k^{(m)}} \sum_{i=1}^{ t_k^{(i)} } \Vert T_{i} x_m \Vert =\infty;
              \]
        \item $\text{span} \{x_1, x_2, \ldots, x_n\} $ is an absolutely mean irregular manifold.
    \end{enumerate}
    By Claim \ref{$P(s_k)$ is residual},
    \[
        X_{n+1}:= \bigcap_{j=1}^n P(s_k^{(n,j)}) \cap P(t_k^{(n)}) \cap D
    \]
    is also a dense $G_{\delta}$ subset of $X$.
    Pick $x_{n+1}\in X_{n+1}$ with $\Vert x_{n+1} -z_{n+1} \Vert <\frac{1}{n+1}$.
    For $j=1, \ldots, n$, there exist a subsequence $(s_k^{(n+1, j)})_k$ of $(s_k^{(n, j)})_k$ such that
    \[
        \lim_{k\to \infty} \frac{1}{s_k^{(n+1, j)}} \sum_{i=1}^{s_k^{(n+1, j)}} \Vert T_{i} x_{n+1} \Vert =0,
    \]
    a subsequence $(s_k^{(n+1, n+1)})_k$ of $(t_k^{(n)})_k$ such that
    \[
        \lim_{k\to \infty} \frac{1}{s_k^{(n+1, n+1)}} \sum_{i=1}^{s_k^{(n+1, n+1)}} \Vert T_{i} x_{n+1} \Vert =0,
    \]
    and an increasing sequence $(t_k^{(n+1)})_k$ in $\mathbb{N}$ such that
    \[
        \lim_{k\to \infty} \frac{1}{t_k^{(n+1)}} \sum_{i=1}^{t_k^{(n+1)}} \Vert T_{i} x_{n+1} \Vert =\infty.
    \]
    For any $\sum_{l=1}^{n+1} \alpha_lx_l \in \text{span} \{x_1, \ldots, x_n, x_{n+1}\} \backslash \{{\bf 0}\}$,
    \[
        \liminf_{k\to \infty} \frac{1}{s_k^{(n+1,1)}} \sum_{i=1}^{s_k^{(n+1,1)} } \Vert T_{i}  (\sum_{l=1}^{n+1} \alpha_lx_l )\Vert \le \sum_{l=1}^{n+1} |\alpha_l|  \lim_{k\to \infty} \frac{1}{s_k^{(n+1,1)}} \sum_{i=1}^{s_k^{(n+1,1)} } \Vert T_{i}  x_l \Vert =0.
    \]
    Let $l'=\min \{l\in \{1, \ldots, n+1\}: \alpha_l \neq 0\}$.
    If $l'<n+1$, then
    \begin{align*}
            & \limsup_{k\to \infty} \frac{1}{s_k^{(n+1,l'+1)}} \sum_{i=1}^{s_k^{(n+1,l'+1)} } \Vert T_{i} (\sum_{l=l'}^{n+1} \alpha_lx_l) \Vert  \\
        \ge &
        |\alpha_{l'}| \lim_{k\to \infty}  \frac{1}{s_k^{(n+1, l'+1)}} \sum_{i=1}^{s_k^{(n+1,l'+1)} } \Vert T_{i} x_{l'} \Vert
        -\sum_{l=l'+1}^{n+1} |\alpha_l| \limsup_{k\to \infty}\frac{1}{s_k^{(n+1,l'+1)}} \sum_{i=1}^{s_k^{(n+1,l'+1)} } \Vert T_{i} x_l \Vert \\
        =   & \infty-0=\infty.
    \end{align*}
    If $l'=n+1$, then by the construction one has
    \[
        \lim_{k\to \infty} \frac{1}{t_k^{(n+1)}} \sum_{i=1}^{t_k^{(n+1)}} \Vert T_{i} x_{n+1} \Vert =\infty.
    \]
    So $\text{span}\{x_1, \ldots, x_n, x_{n+1}\}$ is an absolutely mean irregular manifold.
    By induction, we obtain that the subspace
    $\text{span}\{x_m: m\in \mathbb{N}\}$ is a dense absolutely mean irregular manifold.
\end{proof}

Using Theorem \ref{thm:dense-mean-irregular-manifold}, we have the following characterization of mean Li-Yorke chaos for a sequence of multiples of iterations of the backward shift on $\ell^1(\mathbb{N})$.

\begin{prop}
    Let $X=\ell^1(\mathbb{N})$ and $B\colon X\to X$, $(x_i)_{i=1}^\infty \mapsto (x_i)_{i=2}^\infty$.
    For each $i\in\mathbb{N}$, let $T_i=\lambda_i B^i$, $\lambda_i\in \mathbb{R}$.
    Then the following assertions are equivalent:
    \begin{enumerate}
        \item \label{prop1: mean LY chaotic}
             $(T_i)_i$ is mean Li-Yorke chaotic;
        \item \label{prop2: manifold}
             $(T_i)_i$ has a dense absolutely mean irregular manifold;
        \item \label{prop3: lambda unbounded}
               $\displaystyle
                  \limsup_{n\to\infty} \frac{1}{n} \sum_{i=1}^n |\lambda_i|=\infty$.
    \end{enumerate}
\end{prop}
\begin{proof}
     $(\ref{prop2: manifold})  \Rightarrow  (\ref{prop1: mean LY chaotic})$. It is clear.

     $(\ref{prop1: mean LY chaotic})  \Rightarrow (\ref{prop3: lambda unbounded}) $. 
    Otherwise, we assume that
    \[
        \limsup_{n\to\infty} \frac{1}{n} \sum_{i=1}^n |\lambda_i|=C<\infty.
    \]
    Since $(T_i)_i$ is mean Li-Yorke chaotic, one has $C\neq 0$.
    For each $x=(x_1, x_2, \ldots)\in \ell^1(\mathbb{N})$, we have $\sum_{i=1}^{\infty}|x_i|<\infty$.
    Thus
    \[
        \lim_{n\to \infty} \sum_{i=n}^{\infty} |x_i|=0.
    \]
    Fix $\varepsilon >0$.
    Then there exists $N_0\in \mathbb{N}$ such that for all $n>N_0$, one has
    \[
        \frac{1}{n} \sum_{i=1}^n |\lambda_i|<C+\varepsilon
    \]
    and
    \[
        \sum_{i=n}^{\infty} |x_i|<\frac{\varepsilon}{C}.
    \]
    Thus
    \begin{align*}
        \limsup_{n\to \infty}\frac{1}{n} \sum_{i=1}^n \Vert T_ix \Vert
         & =\limsup_{n\to \infty}\frac{1}{n} \sum_{i=1}^n \Vert \lambda_i B^ix \Vert
        =\limsup_{n\to \infty}\frac{1}{n} \sum_{i=1}^n (|\lambda_i| \sum_{j=i+1}^{\infty} |x_j| ) \\
         & =\limsup_{n\to \infty}\frac{1}{n} \biggr(\sum_{i=1}^{N_0} (|\lambda_i| \sum_{j=i+1}^{\infty} |x_j| )+\sum_{i=N_0+1}^{n} (|\lambda_i| \sum_{j=i+1}^{\infty} |x_j| ) \biggr)   \\
         & \le \limsup_{n\to \infty}\frac{1}{n} \biggr(\sum_{i=1}^{N_0} ( |\lambda_i| \Vert x\Vert)   +\sum_{i=N_0+1}^{n} (|\lambda_i| \frac{\varepsilon}{C} ) \biggr) \\ 
         &\le  \varepsilon+ \frac{\varepsilon^2}{C}.
    \end{align*}

    By the arbitrariness of $\varepsilon$, we have $\limsup_{n\to \infty}\frac{1}{n} \sum_{i=1}^n \Vert T_ix \Vert=0$.
    This implies that there is no absolutely mean semi-irregular vector.
    By Proposition \ref{mean Li-Yorke-pair-vector},
    $(T_i)_{i}$ is not mean Li-Yorke chaotic, which is a contradiction.

    $(\ref{prop3: lambda unbounded})  \Rightarrow (\ref{prop2: manifold})$. 
    Let $X_0$ be the subspace of $X$ consisting of vectors with only finite non-zero coordinates. It is clear that $X_0\times X_0\subset \textrm{MAsym}(T_i)$. Then $\textrm{MAsym}(T_i)$ is dense in $X\times X$.
    For every $k\in\mathbb{N}$, let $y_k=(0,0,\dotsc,0,1 ((k+1)\text{th coordinate)},0,0,\dotsc)$.
    Then $(y_k)_k$ is a bounded sequence in $X$ and
    \[
        \sup_{k}\frac{1}{k}\sum_{i=1}^k \Vert T_i y_k\Vert =
        \sup_{k}\frac{1}{k}\sum_{i=1}^k |\lambda_i|=\infty.
    \]
    By Theorem~\ref{thm:mean-sensitive}, $(T_i)_i$ is mean sensitive.
    Now according to Theorem~\ref{thm:dense-mean-irregular-manifold} $(T_i)_i$ has a dense absolutely mean irregular manifold.
\end{proof}

\section{Mean Li-Yorke chaos for a submultiplicative sequence of operators}\label{sec:LYsub}

In this section, we want to give more characterizations of  mean Li-Yorke chaos and dense mean Li-Yorke chaos for a sequence
of bounded linear operators on a Banach space.
To this end, we need the submultiplicative and almost-commuting properties for a sequence of operators.

Recall that a sequence $(c_i)_{i=1}^{\infty}$ of positive numbers is \emph{submultiplicative} if
\[
    c_{n+m}\leq c_n c_m,\ n,m=1,2,\dotsc.
\]
Inspired by submultiplicative sequence, we introduce the following submultiplicative property for a seuqence of operators.

\begin{defn}
    We say that a sequence $(T_i)_{i=1}^\infty$ of bounded linear operators on a Banach space $X$ is \emph{submultiplicative} if
    there exists a constant $C>0$ such that
    \[
        \Vert T_{i+m}z\Vert \leq C\Vert T_i T_m z \Vert,\
        \forall z\in X \ \text{and} \ i,m\in\mathbb{N}. \tag{$*$}
    \]
\end{defn}
If $T$ is an operator on $X$ and $(\lambda_i)_i$ is a submultiplicative sequence, then the sequence $(\lambda_i T^i)_i$ of operators is submultiplicative.

\begin{lem}\label{lem:mprox-masym}
    Let $(T_i)_{i=1}^\infty$ be a submultiplicative sequence of bounded linear operators on a Banach space $X$.
    If $(T_i)_{i=1}^\infty$ is mean equicontinuous, then every proximal pair is mean asymptotic.
\end{lem}
\begin{proof}
    Since $(T_i)_{i}$ is mean equicontinuous, for any $\varepsilon>0$, there exists $\delta>0$ such that for any $x',y'\in X$ with $\Vert x'-y'\Vert<\delta$, we have
    \[
        \limsup_{n\to\infty}\frac{1}{n}\sum_{i=1}^{n}\Vert T_ix'-T_iy'\Vert<\varepsilon.
    \]
    For any proximal pair $(x,y)$,  there exists $k>0$ such that $\Vert T_kx-T_ky\Vert<\delta$. Then we have
    \begin{align*}
        \limsup_{n\to\infty}\frac{1}{n}\sum_{i=1}^n\Vert T_ix-T_iy\Vert
        =    & \limsup_{n\to\infty}\frac{1}{n}\sum_{i=1}^n\Vert T_{i+k}x-T_{i+k}y\Vert  \\
        =    & \limsup_{n\to\infty}\frac{1}{n}\sum_{i=1}^n\Vert T_{i+k}(x-y)\Vert \\
        \leq & \limsup_{n\to\infty}\frac{1}{n}\sum_{i=1}^nC\Vert T_{i}T_k(x-y)\Vert\quad\text{(by ($*$))} \\
        <    & C\varepsilon.
    \end{align*}
    The arbitrariness of $\varepsilon$ shows that
    \[
        \lim_{n\to\infty}\frac{1}{n}\sum_{i=1}^n\Vert T_ix-T_iy\Vert=\limsup_{n\to\infty}\frac{1}{n}\sum_{i=1}^n\Vert T_ix-T_iy\Vert=0,
    \]
    showing that $(x,y)$ is mean asymptotic.
\end{proof}

\begin{coro}
    Let $(T_i)_{i=1}^\infty$ be a submultiplicative sequence of operators on a Banach space $X$.
    If there exists an absolutely mean semi-irregular vector, then $(T_i)_{i=1}^\infty$ is mean sensitive.
\end{coro}
\begin{proof}
    Otherwise, we assume that $(T_i)_{i}$ is not mean sensitive, by Theorem \ref{thm:dich-mean-eq-mean-sen}, $(T_i)_{i}$ is mean equicontinuous.
    Let $x \in X$ be an absolutely mean semi-irregular vector, then $(x,{\bf 0})$ is a mean proximal pair.
    By Lemma \ref{lem:mprox-masym}, $(x,{\bf 0})$ is mean asymptotic, that is
    \[
        \limsup_{n\to \infty}\frac{1}{n}\sum_{i=1}^n \Vert T_ix \Vert =\lim_{n\to \infty}\frac{1}{n}\sum_{i=1}^n \Vert T_ix \Vert =0,
    \]
    which is a contradiction.
    Thus $(T_i)_{i=1}^\infty$ is mean sensitive.
\end{proof}

We have the following characterization of  dense mean Li-Yorke chaos of a sequences of bounded linear operators with the submultiplicative property.

\begin{thm}\label{thm:equi-dmlyc}
    Let $(T_i)_{i=1}^{\infty}$ be a submultiplicative sequence of bounded linear operators on a separable Banach space $X$.
    Then the following assertions are equivalent:
    \begin{enumerate}
        \item \label{thm1: densely mean LY chaotic}
             $(T_i)_{i}$ is densely mean Li-Yorke chaotic;
        \item \label{thm2: dense mean LY scramled pairs}
             $(T_i)_{i}$ has a dense set of mean Li-Yorke pairs;
        \item \label{thm3: dense mean semi-irr}
             $(T_i)_{i}$ has a dense set of absolutely mean semi-irregular vectors;
        \item \label{thm4: residual mean irr}
              $(T_i)_{i}$ has a residual set of absolutely mean irregular vectors;
        \item \label{thm5: MProx dense-mean sensitive}
              the mean proximal relation of $(T_i)_{i}$ is dense in $X\times X$ and $(T_i)_{i}$ is mean sensitive;
        \item \label{thm: MProx cell dense-mean semi-irregular vector}
        the mean proximal cell of $\mathbf{0}$ is dense in $X$ and there exists a vector $x\in X$ such that
              \[
                  \limsup_{n\to\infty}\frac{1}{n}\sum_{i=1}^{n}\Vert T_ix\Vert>0.
              \]
    \end{enumerate}
\end{thm}
\begin{proof}
     $(\ref{thm1: densely mean LY chaotic})  \Rightarrow (\ref{thm2: dense mean LY scramled pairs})  \Rightarrow (\ref{thm3: dense mean semi-irr})  \Rightarrow  (\ref{thm: MProx cell dense-mean semi-irregular vector}) $.  Those are clear.

    $(\ref{thm5: MProx dense-mean sensitive})   \Rightarrow (\ref{thm4: residual mean irr})    \Rightarrow  (\ref{thm1: densely mean LY chaotic}) $.  They follow from Theorem~\ref{thm:denseLi-Yorke-delta-chaos}.

     $(\ref{thm: MProx cell dense-mean semi-irregular vector})  \Rightarrow  (\ref{thm5: MProx dense-mean sensitive})$. Since $\mathrm{MProx}(T_i,\mathbf{0})$ is dense in $X$, by Lemma \ref{observations 1} $(\ref{MProx dense})$ we know that $\mathrm{MProx}(T_i)$ is dense in $X\times X$. And by Lemma \ref{mean proximal} $(\ref{Lem: MProx cell G-delta})$, $\mathrm{MProx}(T_i,\mathbf{0})$ is a residual set. If $\mathrm{MProx}(T_i,\mathbf{0})=\mathrm{MAsym}(T_i,\mathbf{0})$, then by Lemma \ref{observations 1} $(\ref{MAsym residual})$ we know that $\mathrm{MAsym}(T_i,\mathbf{0})=X$, which contradicts $(\ref{thm: MProx cell dense-mean semi-irregular vector})$. Hence there is an absolutely mean semi-irregular vector in $\mathrm{MProx}(T_i,\mathbf{0})$. By Lemma \ref{lem:mprox-masym}, $(T_i)_{i}$ is not mean equicontinuous, thus $(T_i)_{i}$ is mean sensitive by Theorem \ref{thm:dich-mean-eq-mean-sen} .
\end{proof}

\begin{rem}\label{rem:separability-of-X-2}
According to Remark~\ref{rem:separability-of-X},
(\ref{thm2: dense mean LY scramled pairs})-(\ref{thm: MProx cell dense-mean semi-irregular vector}) of Theorem~\ref{thm:equi-dmlyc}
are also equivalent in the condition that $X$ is a Banach space.
\end{rem}

In \cite{B2000}, Bernal-Gonz\'{a}lez introduced the following almost-commuting property for a sequence of operators.

\begin{defn}
    We say that a sequence $(T_i)_{i=1}^\infty$ of bounded linear operators on a Banach space $X$ is \emph{almost-commuting} if
    \[
        \lim_{i\to\infty} (T_iT_k-T_kT_i)x=0, \ \forall x\in X, k\in\mathbb{N}.
    \]
\end{defn}

\begin{prop}\label{prop:inv}
    Let $(T_i)_{i=1}^\infty$ be an almost-commuting sequence of bounded linear operators on a Banach space $X$ and $(N_n)_n$ be a sequence in $\mathbb{N}$.
    Then the collection
    \[
        X_0= \biggl\{x\in X\colon \lim_{n\to\infty}\frac{1}{N_n}\sum_{i=1}^{N_n}\Vert T_i x\Vert =0\biggr\}
    \]
    is a $(T_i)_{i=1}^\infty$-invariant subspace of $X$.
\end{prop}
\begin{proof}
    For every $x, y\in X_0$, we have
    \begin{align*}
        \limsup_{n\to \infty} \frac{1}{N_n}\sum_{i=1}^{N_n} \Vert T_i (x-y) \Vert & =\limsup_{n\to \infty} \frac{1}{N_n}\sum_{i=1}^{N_n} \Vert T_i x-T_iy \Vert   \\
        & \le \lim_{n\to \infty} \frac{1}{N_n}\sum_{i=1}^{N_n} \Vert T_i x \Vert+\lim_{n\to \infty} \frac{1}{N_n}\sum_{i=1}^{N_n} \Vert T_iy  \Vert \\
        & =0,
    \end{align*}
    and it is clearly that $\lambda x\in X_0$ for every $\lambda \in \mathbb{K}$.
    Thus $X_0$ is a subspace of $X$.

    For every $x\in X_0$ and $k\in \mathbb{N}$,
    \begin{align*}
        \lim_{n\to \infty} \frac{1}{N_n}\sum_{i=1}^{N_n} \Vert T_i (T_kx) \Vert & =\lim_{n\to \infty} \frac{1}{N_n}\sum_{i=1}^{N_n} \Vert T_iT_kx-T_kT_ix+ T_kT_ix\Vert \\
      & \le \lim_{n\to \infty} \frac{1}{N_n}\sum_{i=1}^{N_n} \Vert T_iT_kx-T_kT_ix \Vert+\lim_{n\to \infty} \frac{1}{N_n}\sum_{i=1}^{N_n} \Vert T_kT_ix \Vert  \\
      & \le \lim_{n\to \infty} \frac{1}{N_n}\sum_{i=1}^{N_n} \Vert T_iT_kx-T_kT_ix \Vert+ \Vert T_k \Vert  \lim_{n\to \infty} \frac{1}{N_n}\sum_{i=1}^{N_n} \Vert  T_ix \Vert \\
       & =0,
    \end{align*}
    the last equation follows from that $(T_i)_i$ is alomost-commuting.
    Then $T_kx\in X_0$, which implies that $X_0$ is $(T_i)_{i}$-invariant.
\end{proof}

We have the following characterization of mean Li-Yorke chaos for a sequence of bounded linear operators with the almost-commuting and submultiplicaitve properties.

\begin{thm}\label{thm:euqi-lmyc}
    Let $(T_i)_{i=1}^{\infty}$ be an almost-commuting, submultiplicative sequence of bounded linear operators on a Banach space $X$.
    Then the following assertions are equivalent:
    \begin{enumerate}
        \item \label{thm: mean LY scrambled pair}
               $(T_i)_{i}$ has a  mean Li-Yorke pair;
        \item \label{thm: mean semi-irregular vector}
              $(T_i)_{i}$ has an  absolutely mean semi-irregular vector;
        \item \label{thm: mean LY chaotic}
              $(T_i)_{i}$ is mean Li-Yorke chaotic;
        \item \label{thm: mean irregular vector}
              $(T_i)_{i}$ has an  absolutely mean irregular vector;
        \item \label{thm: invariant space-residual mean irregular vectors}
              there exists a $(T_i)_{i}$-invariant closed subspace $X_0$ of $X$ such that
              $X_0$ has a residual subset of absolutely mean irregular vectors.
    \end{enumerate}
\end{thm}
\begin{proof} 
   $(\ref{thm: mean LY scrambled pair})  \Leftrightarrow (\ref{thm: mean semi-irregular vector})   \Leftrightarrow  (\ref{thm: mean LY chaotic})$. Those follow from Proposition \ref{mean Li-Yorke-pair-vector}.  
   
     $ (\ref{thm: invariant space-residual mean irregular vectors}) \Rightarrow  (\ref{thm: mean irregular vector})  \Rightarrow (\ref{thm: mean semi-irregular vector}). $ Those are clear.
     
     $(\ref{thm: mean semi-irregular vector}) \Rightarrow (\ref{thm: invariant space-residual mean irregular vectors}) $. 
     Let $x_0$ be an absolutely mean semi-irregular vector. Then there exists an increasing sequence of integers $(n_k)_k$ such that
    \[
        \lim_{k\to\infty}\frac{1}{n_k}\sum_{i=1}^{n_k}\Vert T_ix_0\Vert=0.
    \]
    Define
    \[
        X_0=\left\{y\in X\colon \lim_{k\to\infty}\frac{1}{n_k}\sum_{i=1}^{n_k}\Vert T_i y\Vert=0\right\}.
    \]
     As $(T_i)_{i=1}^\infty$ is almost-commuting, by Proposition \ref{prop:inv} $X_0$ is a $(T_i)_{i}$-invariant subspace of $X$.
    Let $\widetilde{X}_0$ be the closure of $X_0$.
    Then $\widetilde{X}_0$ is a $(T_i)_{i}$-invariant closed subspace of $X$.

    Clearly $x_0\in \widetilde{X}_0$ and 
    \[
        \limsup_{n\to\infty}\frac{1}{n}\sum_{i=1}^n\Vert T_ix_0\Vert>0.
    \]
    By the construction of $\widetilde{X}_0$, the mean proximal cell of $\mathbf{0}$ contains $X_0$, and then is dense in $\widetilde{X}_0$. Now by Theorem \ref{thm:equi-dmlyc} and Remark~\ref{rem:separability-of-X-2} , we know that $(T_i|_{\widetilde{X}_0})_{i=1}^\infty$ has a residual set of absolutely mean irregular vectors.
\end{proof}

By the proof of Theorem~\ref{thm:euqi-lmyc}, we have the following consequence.
\begin{coro}
    Let $(T_i)_{i=1}^{\infty}$ be an almost-commuting, submultiplicative sequence of bounded linear operators on a Banach space $X$.
    Then the set of all absolutely mean irregular vectors is dense in the set of all absolutely mean semi-irregular vectors.
\end{coro}

In \cite{BBP2020}, Bernardes et al. introduced the mean Li-Yorke chaos criterion for an operator. Here we generalize the concept to a sequence of operators.

\begin{defn}
    Let $(T_i)_{i=1}^\infty$ be a sequence of bounded linear operators on a Banach space $X$.
    We say that $(T_i)_{i=1}^\infty$ satisfies the \emph{mean Li-Yorke chaos criterion} if there exists a subset $X_0$ of $X$ with the following properties:
    \begin{enumerate}
        \item for every $x\in X_0$, $\liminf_{n\to \infty} \frac{1}{n} \sum_{i=1}^n \Vert T_ix \Vert=0$;
        \item there exist a sequence $(y_k)_k$ in $\overline{ \text{span}(X_0)}$ and a sequence $(N_k)_k$ in $\mathbb{N}$ such that for every $k\in \mathbb{N}$,
              \[
                  \frac{1}{N_k} \sum_{i=1}^{N_k} \Vert T_iy_k \Vert \ge k \Vert y_k \Vert.
              \]
    \end{enumerate}
\end{defn}

\begin{thm}
    Let $(T_i)_{i=1}^{\infty}$ be an almost-commuting, submultiplicative sequence of bounded linear operators on a Banach space $X$.
    Then $(T_i)_{i=1}^\infty$ is mean Li-Yorke chaotic if and only if it satisfies the mean Li-Yorke chaos criterion.
\end{thm}
\begin{proof}
    Suppose that $(T_i)_{i=1}^\infty$ satisfies the mean Li-Yorke chaos criterion. Clearly, $\overline{\mathrm{span}(X_0)}$ is a closed subspace of $X$. Let $\tilde{y}_k=\frac{1}{\Vert y_k\Vert}y_k$. Clearly $(\tilde{y}_k)_k$ is a bounded sequence and $\sup_k\frac{1}{N_k}\sum_{i=1}^{N_k}\Vert T_i\tilde{y}_k\Vert=\infty$. By Theorem \ref{thm:mean-sensitive} we know that $\left(T_i|_{\overline{\mathrm{span}(X_0)}}\right)_{i=1}^\infty$ is mean sensitive. Noticing that for any $y\in\mathrm{span}(X_0)$, we have
    \[
        \liminf_{n\to\infty}\frac{1}{n}\sum_{i=1}^n\Vert T_iy\Vert=0.
    \]
    Hence we have $\mathrm{MProx}\left( T_i|_{\overline{\mathrm{span}(X_0)}},\mathbf{0}\right)$ is dense in $\overline{\mathrm{span}(X_0)}$.
    By Theorem \ref{thm:equi-dmlyc}, we know that $\left(T_i|_{\overline{\mathrm{span}(X_0)}}\right)_{i=1}^\infty$  is mean Li-Yorke chaotic. Then $(T_i)_{i=1}^\infty$ is also mean Li-Yorke chaotic.

    Now suppose that $(T_i)_{i=1}^\infty$ is mean Li-Yorke chaotic. By Theorem \ref{thm:euqi-lmyc}, there exists a closed subspace $\widetilde{X}$ which has a residual subset of absolutely mean irregular vectors. Let $X_0=\mathrm{MProx}(T_i|_{\widetilde{X}},\mathbf{0})$. By Theorem \ref{thm:equi-dmlyc} and Theorem \ref{thm:mean-sensitive}, we know that $X_0$ satisfies the requirements of the mean Li-Yorke chaos criterion.
\end{proof}

\begin{rem}
    It is worth noting that in Section 7.3 of \cite{GP2011} discretizations of $C_0$-semigroups are considered, which certainly fulfill the commutativity condition, and under general conditions also fulfill the submultiplicative property. 
    See \cites{ABMP2013, CKMM2016, Wu2014, ZY2023} for the research on Li-Yorke and distributional chaos for $C_0$-semigroups.
    Therefore, the results in Section 4 can be applied to sequences of operators which are discretizations of $C_0$-semigroups. 
\end{rem}

\medskip 

\noindent \textbf{Acknowledgment.}
J. Li was suported in part by NSF of China (Grant no.~12222110).
X. Wang was supported in part by NSF of China (Grant no.~12301230) and STU Scientific Research Initiation Grant (SRIG, NTF22020).
J. Zhao (corresponding author) is supported by NSF of China (Grant no.~12301226).
The authors  would like to thank the referees for the helpful suggestions.

\end{document}